%% file: main.tex
\documentclass{amsart}

\usepackage{mathrsfs}

\input{preamble}

\renewcommand{\bl}[1]{#1}

\title[Fourth-Order Moment Operators and Gaussian Measure Singularity]{An Operator-Theoretic Characterization of Gaussian Measure Singularity under Mean Shifts}

\author{Marc Vidal}
\address{Department of Data Analysis, Ghent University}
\email{marc.vidalbadia@ugent.be}


\begin{document}

\begin{abstract}
For two Gaussian random elements on a separable Hilbert space with common covariance operator, the fourth-order moment tensor structure of their mixture admits an operator-valued representation obtained from the purely quadratic component of the class-conditional second moment. We identify the tensor mechanism generating the completely diagonal coefficients of this fourth-order representation after covariance standardization, and show that these coefficients are governed entirely by the coordinatewise Cameron–Martin energy of the mean shift. This yields a necessary and sufficient condition, together with an explicit eigenvalue, for the Fisher discriminant to be an eigenfunction of the induced coordinatewise operator. We further introduce aggregated fourth-order spectral functionals for two complementary constructions, one built from fourth-order moments and the other from a product-of-expectations counterpart, and prove that they are asymptotically equivalent precisely when the underlying Gaussian measures are mutually singular. These results provide a spectral realization of the classical Cameron–Martin criterion through higher-order moment operators, explaining the probabilistic origin of fourth-order spectral quantities previously proposed for Gaussian discrimination and functional data classification and relating them to the ``near-perfect classification’’ regime.
\end{abstract}

\maketitle

\section{Introduction}\label{sec:introduction}
One of the fundamental phenomena in the theory of Gaussian processes on infinite-dimensional Hilbert spaces is that Gaussian measures are either equivalent or mutually singular, a dichotomy that was established independently by Feldman \cite{Feldman58} and H\'ajek \cite{Hajek58}. Shepp \cite{Shepp64,Shepp66,SheppRN66} subsequently developed finite-dimensional criteria for deciding between the two alternatives, reformulated the dichotomy through Hellinger and Jeffreys functionals, and derived explicit Radon–Nikodym derivatives by separating covariance and mean contributions, leading respectively to what are now recognized as Hilbert–Schmidt conditions on the covariance operators and the classical Cameron–Martin condition on the mean shift.
Statistical implications were investigated by Rao and Varadarajan \cite{RaoVaradarajan63} in the context of discrimination between Gaussian processes. For Gaussian measures with common covariance operator, the Cameron–Martin space (see, e.g.,  \cite{Janson97}) characterizes the admissible directions of mean shifts preserving equivalence. Although Gaussian sample paths almost surely lie outside this space, it provides the natural geometric setting for studying translations and equivalence of Gaussian measures. These developments were substantially generalized by Luki\'c and Beder \cite{Lukic01}, who proposed a reproducing kernel Hilbert space (RKHS) framework for stochastic processes based on dominance and nuclear dominance, establishing the conditions for Gaussian processes to admit sample paths in prescribed RKHSs. More recently, Santoro, Waghmare and Panaretos \cite{Santoro26} showed that kernel covariance embeddings provide an operator-valued representation that recasts arbitrary two-sample problems as discrimination problems between Gaussian measures through the Feldman–Hájek dichotomy.

The separation of measure phenomenon has also re-emerged in functional data analysis through the near-perfect classification paradigm \cite{Delaigle12,Berrendero18}, where the increasing information carried by infinite-dimensional observations allows the misclassification probability to converge to zero asymptotically. Existing methodologies typically exploit the Cameron–Martin geometry indirectly through finite-dimensional approximations and regularization. More recently, fourth-order moment operators have been proposed for binary discrimination \cite{Pena14,Vidal25,Vidal25moment} and later extended to multiclass classification of functional data \cite{Vidal26IFCS}. %Although these works identified empirical associations between the spectrum of fourth-order moment operators and the equivalence–singularity dichotomy, the probabilistic origin of this phenomenon remained unclear. 
\bl{For a single discriminant projection under a common covariance operator, the link between the fourth-order spectrum and the equivalence--singularity dichotomy has been established in \cite{Vidal25} (Corollary~3.1 of the corrected version). What remains unexplained is how this geometry is encoded in the fourth-order tensor structure itself.} The present work shows that the Fisher discriminant is naturally associated with the spectrum of two complementary fourth-order operator constructions whose spectral quantities are completely determined by the Cameron–Martin geometry, thereby providing an operator-theoretic bridge between Gaussian measure theory and functional discrimination. In particular, the fourth-order spectrum \bl{of these operator constructions} provides a probabilistic signature of the near-perfect classification regime.

Higher moments of Banach-valued random variables have been systematically studied by Janson and Kaijser \cite{Janson15} through tensor products. In particular, they showed that the second moment admits a natural representation as a linear operator on Hilbert spaces, while noting that it is unknown whether an analogous construction exists for higher-order moments. Motivated by this observation, Vidal \cite{Vidal25moment} proposed an alternative operator construction based on compositions of rank-one operators induced by second-order tensors, which may be interpreted as contractions of the corresponding tensor moments. This approach generalizes the natural operator representation of the second moment while remaining amenable to spectral analysis. Under Gaussian mixture models, it was shown that the Fisher discriminant is associated with a spectral subspace of these higher-order moment operators, thereby recovering, %through an explicit operator representation, 
a result previously established by Pe\~na, Prieto and Rend\'on \cite{Pena14}. In their pioneering work, the authors introduced a kurtosis operator together with a covariance-dependent transformation satisfying a suitable equivariance condition. Exploiting the rotational invariance of the transformed Gaussian model, they showed that the Fisher discriminant direction can be recovered as an eigendirection. The proof, however, is indirect, relying on invariance arguments rather than on a characterization of the underlying fourth-order structure or of the spectral mechanism giving rise to the discriminant eigendirection. 
While Vidal \cite{Vidal25moment} recovered the same eigenvalue equation through direct operator calculations, it did not explain how this equation arises from the underlying fourth-order tensor structure\bl{, nor did it establish the precise conditions under which the discriminant direction is an eigenfunction of the resulting operator}. The present work addresses this question by identifying the tensor mechanism underlying the corresponding coordinatewise discriminant eigenstructure \bl{--- formalized below as a structured component of the complete fourth-order moment tensor ---} and \bl{by establishing the precise conditions under which the associated eigenfunction property holds}, thereby providing the basis for the spectral analysis developed below. \bl{The complete fourth-order moment tensor is treated separately, later in the paper (Proposition~\ref{prop:full-tensor}), as a benchmark for distinguishing properties specific to the structured construction from those that persist at the level of the full fourth-order tensor.}

\subsection{Contributions}
The main contributions of this paper are threefold.
First, we establish an explicit operator-theoretic connection between fourth-order moment operators and the Cameron--Martin geometry by characterizing the spectrum of the induced coordinatewise fourth-order operator in terms of coordinatewise Cameron--Martin norms (see Lemma~\ref{lem:diagonal-coefficients} and Theorem~\ref{T1}).
Second, we introduce aggregated fourth-order spectral quantities associated with two complementary fourth-order operator constructions and prove that their discriminative spectra are asymptotically equivalent in the singular regime (Theorem~\ref{thm:asymptotic-equivalence}).
Finally, we show that the divergence of these aggregated spectral quantities is equivalent to the Feldman--H\'ajek singularity criterion under a common covariance operator, thereby providing an operator-theoretic realization of the Cameron--Martin characterization of Gaussian measure singularity (Theorem~\ref{T2}).

\section{Preliminaries}
\subsection{Gaussian mixtures in Hilbert spaces}
Let
$H:= L^2([0,1])$,
equipped with the usual inner product
$\langle\cdot,\cdot\rangle_H$
and norm
$\|\cdot\|_H$.
We consider observations arising from a mixture of two Gaussian random elements
$X_1,X_2$ taking values in $H$ and satisfying
$X_\ell\sim\mathbb G(m_\ell,\mathsf K)$,
$\ell=1,2$,
where both populations share the same covariance operator
$\mathsf K:H\rightarrow H$,
while their mean functions
$m_1,m_2\in H$
may differ. An observation from the mixture is generated according to
\[
X=
\begin{cases}
X_1,&\text{with probability }\alpha,\\
X_2,&\text{with probability }1-\alpha,
\end{cases} 
\qquad
0<\alpha<1.
\]
The common covariance operator is assumed to be induced by a continuous covariance kernel
$\mathfrak K(s,t)$ through
\[
(\mathsf Kf)(t)
=
\int_0^1 \mathfrak K(s,t)f(s)\,ds,
\]
and is therefore compact, self-adjoint, positive and trace class \cite[see Theorems 5.2 and 5.4]{Janson15}.
By the spectral theorem, there exists an orthonormal basis
$\{\gamma_j\}_{j\ge1}$ of $H$
and a non-increasing sequence of non-negative eigenvalues
$\lambda_1\ge\lambda_2\ge\cdots\ge0$
such that
$\mathsf K\gamma_j=\lambda_j\gamma_j$
for every $j$. 
Consequently, each Gaussian component admits the Karhunen--Loève expansion
\[
X_\ell
=
m_\ell
+
\sum_{j=1}^{\infty}
\xi_{\ell j}\gamma_j,
\qquad \ell =1,2,
\]
where
$\xi_{\ell j}=\langle X_\ell-m_\ell,\gamma_j\rangle_H$
are mutually independent Gaussian random variables satisfying
$\xi_{\ell j}\sim N(0,\lambda_j)$ \cite[see, e.g., Sect.~7.3]{HsingEubank15}.

Throughout the paper we assume that $\lambda_j>0$ for every $j$, so that
$\mathsf K$ is injective.
We write the mean difference as
\[
m_\Delta
=
m_2-m_1
=
\sum_{j=1}^{\infty}\nu_j\gamma_j,
\]
where $\nu_j=\langle m_\Delta,\gamma_j\rangle_H$.
The Fisher discriminant criterion for the mixture is defined by
\[
J(\beta)
=
\frac{\alpha(1-\alpha)\langle\beta,m_\Delta\rangle_H^2}
{\langle\beta,\mathsf K\beta\rangle_H}.
\]

Assuming that $\sum_{j=1}^{\infty}\frac{\nu_j^2}{\lambda_j^2}<\infty,$ the Moore--Penrose inverse $\mathsf K^\dagger$ is well defined at
$m_\Delta$ \cite[Ch.~3]{HsingEubank15}, and the Fisher criterion is
maximized, up to a non-zero multiplicative constant, by
\[
\beta
=
\mathsf K^\dagger m_\Delta
=
\sum_{j=1}^{\infty}
\frac{\nu_j}{\lambda_j}\gamma_j.
\]

The above assumption is stronger than the condition characterizing
equivalence of Gaussian measures under a common covariance operator,
which is introduced later through the Cameron--Martin space.
Consequently, the Fisher discriminant may fail to exist as an element of $H$, 
even when the corresponding Gaussian measures are equivalent.
For this reason, the subsequent analysis is carried out through
finite-dimensional Karhunen--Loève truncations, for which the Fisher
discriminant is well defined at every truncation level.

\subsection{Higher-order moment representations}
For $f,g\in H$, the tensor product $f\otimes g\in H\otimes H$ is defined by $(f\otimes g)(h)=\langle h,g\rangle_H f,$ for every $h\in H$, under the canonical identification of $H\otimes H$ with the space of Hilbert--Schmidt operators on $H$. In particular, $f\otimes g$ is a rank-one operator. 

The second-order moment of a centred random element $X$ is therefore naturally represented by $\mathsf K=\mathbb E(X\otimes X)$, which coincides with its covariance operator. Higher-order moments can be analogously defined through tensor products. Let $H^{\otimes p}=H\otimes\cdots\otimes H$ denote the $p$-fold Hilbert tensor product of $H$. The $p$th tensor moment is defined by $\mathbb E(X^{\otimes p})\in H^{\otimes p}$, where $X^{\otimes p}=X\otimes\cdots\otimes X$ is the $p$th tensor power of $X$. When $p=2r$ is even, the tensor space $H^{\otimes 2r}$ admits an equivalent representation in terms of Hilbert--Schmidt operators. Indeed, since $\mathcal L_2(H)\cong H^{\otimes 2}$, iterating this canonical identification yields $\mathcal L_2(H)^{\otimes r}\cong H^{\otimes 2r}$. Hence, under $(H\otimes H)^{\otimes r}\cong H^{\otimes 2r}$, the tensor $(X\otimes X)^{\otimes r}$ is identified with $X^{\otimes 2r}$. Consequently, the corresponding tensor moments coincide under the same identification. By contrast, odd-order moments contain an unpaired copy of $X$ and therefore do not admit a natural representation as operators on $H$. To address this issue, Vidal \cite{Vidal25moment} proposed representing even-order moments through compositions of the rank-one operator $X\otimes X$. Even-order moment operators are thus obtained by repeated compositions of $X\otimes X$\bl{, whereas an extension to odd orders based on fractional powers was also proposed.}

The present work is solely concerned with fourth-order moments. We combine both viewpoints by first constructing fourth-order tensor moments in $H^{\otimes4}$ and subsequently deriving associated operator representations. Although \bl{the fourth-order moment operators considered here are positive self-adjoint Hilbert--Schmidt operators and therefore admit spectral decompositions}, throughout the paper all tensor expansions are expressed with respect to the Karhunen–Loève basis. This covariance-induced representation allows every fourth-order construction to be expressed in terms of the Karhunen–Loève coefficients, reducing the analysis to Gaussian moment identities. Moreover, it provides a natural finite-dimensional approximation that regularizes the inverse covariance operator required for standardization, thereby avoiding the need for an independent spectral decomposition of the fourth-order operators introduced below.

Let $m=\alpha m_1+(1-\alpha)m_2$ and $Y=X-m$. After centring by the mixture mean, write
$a_1=-(1-\alpha)$, $a_2=\alpha$, $\pi_1=\alpha$, and $\pi_2=1-\alpha$.
For each class $\ell\in\{1,2\}$, define
\[
Y_\ell
=
\sum_{j=1}^{\infty}
(\xi_{\ell j}+a_\ell\nu_j)\gamma_j,
\]
where the coefficients $\xi_{\ell j}$ are independent Gaussian random
variables with mean zero and variance $\lambda_j$. Then
\[
Y=
\begin{cases}
Y_1,&\text{with probability }\pi_1,\\
Y_2,&\text{with probability }\pi_2.
\end{cases}
\]

The covariance operator of the centred mixture $Y$ is
\[
\begin{aligned}
\mathsf K_Y(u)
&=
\sum_{\ell=1}^{2}
\pi_\ell\,
\mathbb E
\bigl[
\langle Y_\ell,u\rangle_H\,Y_\ell
\bigr] \\
&=
\mathsf Ku
+
\alpha(1-\alpha)
\langle u,m_\Delta\rangle_H\,m_\Delta,
\qquad u\in H.
\end{aligned}
\]
Hence $\mathsf K_Y=\mathsf K+\alpha(1-\alpha)m_\Delta\otimes m_\Delta$.
Each Gaussian component has finite moments of every order (see, e.g., Fernique's theorem), and the mixture differs only by deterministic mean shifts. Therefore, $\mathbb E\|Y\|_H^4<\infty$. Moreover, $\|Y^{\otimes4}\|_{H^{\otimes4}}=\|Y\|_H^4$, so that $Y^{\otimes4}$ is Bochner integrable in $H^{\otimes4}$. The coefficient expansions below are obtained by applying the continuous coordinate functionals associated with the tensor-product Karhunen--Loève basis to the Bochner integral defining $\mathsf M_Y^{(4)}$, and then expanding the Karhunen--Loève representation of $Y$. Accordingly, the centred fourth-order moment is defined by
\[
\mathsf M_Y^{(4)}
=
\mathbb E\bigl(Y^{\otimes4}\bigr)
=
\sum_{\ell=1}^{2}\pi_\ell
\mathbb E
\left[
\left(
\sum_{j=1}^{\infty}
(\xi_{\ell j}+a_\ell\nu_j)\gamma_j
\right)^{\otimes4}
\right].
\]
Equivalently,
\[
\mathsf M_Y^{(4)}
=
\sum_{j,k,m,n\ge1}
M_{jkmn}
\,
\gamma_j\otimes\gamma_k\otimes\gamma_m\otimes\gamma_n,
\]
where
\[
M_{jkmn}
=
\sum_{\ell=1}^{2}\pi_\ell
\mathbb E
\left[
(\xi_{\ell j}+a_\ell\nu_j)
(\xi_{\ell k}+a_\ell\nu_k)
(\xi_{\ell m}+a_\ell\nu_m)
(\xi_{\ell n}+a_\ell\nu_n)
\right].
\]

Expanding the product gives
\[
\begin{aligned}
M_{jkmn}
=
\sum_{\ell=1}^{2}\pi_\ell
\Bigl\{&
\mathbb E
(\xi_{\ell j}\xi_{\ell k}\xi_{\ell m}\xi_{\ell n})
\\
&+
a_\ell\nu_j\,\mathbb E(\xi_{\ell k}\xi_{\ell m}\xi_{\ell n})
+
a_\ell\nu_k\,\mathbb E(\xi_{\ell j}\xi_{\ell m}\xi_{\ell n})
\\
&+
a_\ell\nu_m\,\mathbb E(\xi_{\ell j}\xi_{\ell k}\xi_{\ell n})
+
a_\ell\nu_n\,\mathbb E(\xi_{\ell j}\xi_{\ell k}\xi_{\ell m})
\\
&+
a_\ell^2\nu_j\nu_k\,\mathbb E(\xi_{\ell m}\xi_{\ell n})
+
a_\ell^2\nu_j\nu_m\,\mathbb E(\xi_{\ell k}\xi_{\ell n})
\\
&+
a_\ell^2\nu_j\nu_n\,\mathbb E(\xi_{\ell k}\xi_{\ell m})
+
a_\ell^2\nu_k\nu_m\,\mathbb E(\xi_{\ell j}\xi_{\ell n})
\\
&+
a_\ell^2\nu_k\nu_n\,\mathbb E(\xi_{\ell j}\xi_{\ell m})
+
a_\ell^2\nu_m\nu_n\,\mathbb E(\xi_{\ell j}\xi_{\ell k})
\\
&+
a_\ell^3\nu_j\nu_k\nu_m\,\mathbb E(\xi_{\ell n})
+
a_\ell^3\nu_j\nu_k\nu_n\,\mathbb E(\xi_{\ell m})
\\
&+
a_\ell^3\nu_j\nu_m\nu_n\,\mathbb E(\xi_{\ell k})
+
a_\ell^3\nu_k\nu_m\nu_n\,\mathbb E(\xi_{\ell j})
\\
&+
a_\ell^4\nu_j\nu_k\nu_m\nu_n
\Bigr\}.
\end{aligned}
\]

Since the Gaussian coefficients are centred,
$\mathbb E(\xi_{\ell j})=0$ and
$\mathbb E(\xi_{\ell j}\xi_{\ell k}\xi_{\ell r})=0$,
all terms containing either one or three Gaussian factors vanish. Therefore,
\[
\begin{aligned}
M_{jkmn}
=
\sum_{\ell=1}^{2}\pi_\ell
\Bigl\{&
\mathbb E
(\xi_{\ell j}\xi_{\ell k}\xi_{\ell m}\xi_{\ell n})
\\
&+
a_\ell^2\Bigl[
\nu_j\nu_k\,\mathbb E(\xi_{\ell m}\xi_{\ell n})
+\nu_j\nu_m\,\mathbb E(\xi_{\ell k}\xi_{\ell n})
+\nu_j\nu_n\,\mathbb E(\xi_{\ell k}\xi_{\ell m})
\\
&\hspace{2.2cm}
+\nu_k\nu_m\,\mathbb E(\xi_{\ell j}\xi_{\ell n})
+\nu_k\nu_n\,\mathbb E(\xi_{\ell j}\xi_{\ell m})
+\nu_m\nu_n\,\mathbb E(\xi_{\ell j}\xi_{\ell k})
\Bigr]
\\
&+
a_\ell^4
\nu_j\nu_k\nu_m\nu_n
\Bigr\}.
\end{aligned}
\]
Under the canonical identification
$
H^{\otimes4}
\cong
\mathcal L_2(H)\otimes\mathcal L_2(H),
$
the fourth-order tensor moment $\mathsf M_Y^{(4)}$ admits an equivalent operator representation. Since the canonical identification is an isometric isomorphism (hence a bounded linear operator), it commutes with Bochner integration. Thus, a fourth-order tensor may be represented as the second moment of a random element taking values in $H^{\otimes2}\cong\mathcal L_2(H)$. \bl{Under the further canonical identification $\mathcal L_2(H)\otimes\mathcal L_2(H)\cong\mathcal L_2(\mathcal L_2(H))$, this second moment is a Hilbert--Schmidt operator on $\mathcal L_2(H)$, not an operator acting on $H$.}

For each class $\ell$, decompose the random rank-one operator as
\[
Y_\ell\otimes Y_\ell
=
\mathsf V_\ell^{(2)}
+
\mathsf R_\ell,
\]
where
\[
\mathsf V_\ell^{(2)}
=
\sum_{j,k\ge1}
\bigl(
\xi_{\ell j}\xi_{\ell k}
+
a_\ell^2\nu_j\nu_k
\bigr)
(\gamma_j\otimes\gamma_k),
\quad
\mathsf R_\ell
=
\sum_{j,k\ge1}
a_\ell
\bigl(
\xi_{\ell j}\nu_k
+
\nu_j\xi_{\ell k}
\bigr)
(\gamma_j\otimes\gamma_k).
\]

Define
\[
\mathsf V^{(4)}
=
\sum_{\ell=1}^{2}
\pi_\ell
\mathbb E
\left(
\mathsf V_\ell^{(2)}
\otimes
\mathsf V_\ell^{(2)}
\right).
\]
Then
\[
\mathsf M_Y^{(4)}
=
\mathsf V^{(4)}
+
\sum_{\ell=1}^{2}\pi_\ell
\mathbb E\!\left(
\mathsf V_\ell^{(2)}\otimes\mathsf R_\ell
\right)
+
\sum_{\ell=1}^{2}\pi_\ell
\mathbb E\!\left(
\mathsf R_\ell\otimes\mathsf V_\ell^{(2)}
\right)
+
\sum_{\ell=1}^{2}\pi_\ell
\mathbb E\!\left(
\mathsf R_\ell\otimes\mathsf R_\ell
\right).
\]

The operator $\mathsf V^{(4)}$ therefore represents the component of the exact fourth-order moment associated with the purely Gaussian second-order structure and the deterministic between-class contribution, whereas the remaining terms collect the interactions generated by the linear mean-shift component. Thus, $\mathsf V^{(4)}$ is not the complete fourth-order moment tensor of the mixture, but a structured fourth-order component obtained by excluding the interaction terms generated by the linear mean shift. This is precisely the construction underlying the moment operators studied in Vidal \cite{Vidal25moment}; it separates the covariance and between-class contributions while retaining a tractable operator-valued fourth-order representation. The subsequent analysis is therefore devoted to $\mathsf V^{(4)}$, whose structure admits an explicit spectral characterization.

Although $\mathsf V^{(4)}$ is defined as a second moment in $H^{\otimes2}$, it represents a fourth-order tensor on the original Hilbert space $H$. Expanding $\mathsf V^{(4)}$ with respect to the tensor-product Karhunen–Loève basis gives
\[
\mathsf V^{(4)}
=
\sum_{j,k,m,n\ge1}
V_{jkmn}\,
\gamma_j\otimes\gamma_k\otimes\gamma_m\otimes\gamma_n,
\]
where
\[
V_{jkmn}
=
\sum_{\ell=1}^{2}
\pi_\ell\,
\mathbb E\!\left[
(\xi_{\ell j}\xi_{\ell k}+a_\ell^2\nu_j\nu_k)
(\xi_{\ell m}\xi_{\ell n}+a_\ell^2\nu_m\nu_n)
\right].
\]
Since $\mathsf V^{(4)}$ excludes the interaction component $\mathsf R_\ell$, expanding the product yields only four terms
\[
\begin{aligned}
V_{jkmn}
={}&
\sum_{\ell=1}^{2}\pi_\ell
\Bigl\{
\mathbb E
(\xi_{\ell j}\xi_{\ell k}\xi_{\ell m}\xi_{\ell n})
\\
&+
a_\ell^2
\nu_j\nu_k\,
\mathbb E(\xi_{\ell m}\xi_{\ell n})
+
a_\ell^2
\nu_m\nu_n\,
\mathbb E(\xi_{\ell j}\xi_{\ell k})
\\
&+
a_\ell^4
\nu_j\nu_k\nu_m\nu_n
\Bigr\}.
\end{aligned}
\]
rather than the sixteen terms arising in the expansion of the exact fourth-order tensor moment. Using Isserlis' formula,
\[
\begin{aligned}
\mathbb E
(\xi_{\ell j}\xi_{\ell k}\xi_{\ell m}\xi_{\ell n})
={}&
\mathbb E(\xi_{\ell j}\xi_{\ell k})
\mathbb E(\xi_{\ell m}\xi_{\ell n})
\\
&+
\mathbb E(\xi_{\ell j}\xi_{\ell m})
\mathbb E(\xi_{\ell k}\xi_{\ell n})
\\
&+
\mathbb E(\xi_{\ell j}\xi_{\ell n})
\mathbb E(\xi_{\ell k}\xi_{\ell m}).
\end{aligned}
\]
Since $\mathbb E(\xi_{\ell j}\xi_{\ell k})=\lambda_j\delta_{jk}$, we obtain
\[
\begin{aligned}
\mathbb E
(\xi_{\ell j}\xi_{\ell k}\xi_{\ell m}\xi_{\ell n})
={}&
\lambda_j\lambda_m\delta_{jk}\delta_{mn}
+
\lambda_j\lambda_k\delta_{jm}\delta_{kn}
\\
&+
\lambda_j\lambda_k\delta_{jn}\delta_{km}.
\end{aligned}
\]
Consequently,
\[
\begin{aligned}
V_{jkmn}
=
\sum_{\ell=1}^{2}\pi_\ell
\Bigl\{&
\lambda_j\lambda_m\delta_{jk}\delta_{mn}
+
\lambda_j\lambda_k\delta_{jm}\delta_{kn}
+
\lambda_j\lambda_k\delta_{jn}\delta_{km}
\\
&+
a_\ell^2
\left[
\nu_j\nu_k\lambda_m\delta_{mn}
+
\nu_m\nu_n\lambda_j\delta_{jk}
\right]
\\
&+
a_\ell^4\nu_j\nu_k\nu_m\nu_n
\Bigr\}.
\end{aligned}
\]
Using
\begin{align}
\sum_{\ell=1}^{2}\pi_\ell a_\ell^2
&=\alpha(1-\alpha)=\rho, \label{eq1}\\
\sum_{\ell=1}^{2}\pi_\ell a_\ell^4
&=\alpha(1-\alpha)^4+(1-\alpha)\alpha^4 \notag\\
&=\alpha(1-\alpha)\bigl[(1-\alpha)^3+\alpha^3\bigr]
=\rho(1-3\rho), \label{eq2}
\end{align}
where $\rho=\alpha(1-\alpha)$,
the coefficients reduce to
\[
\begin{aligned}
V_{jkmn}
={}&
\lambda_j\lambda_m\delta_{jk}\delta_{mn}
+
\lambda_j\lambda_k\delta_{jm}\delta_{kn}
+
\lambda_j\lambda_k\delta_{jn}\delta_{km}
\\
&+
\rho
\Bigl[
\nu_j\nu_k\lambda_m\delta_{mn}
+
\nu_m\nu_n\lambda_j\delta_{jk}
+
(1-3\rho)
\nu_j\nu_k\nu_m\nu_n
\Bigr].
\end{aligned}
\]

Neither $\mathsf M_Y^{(4)}$ nor $\mathsf V^{(4)}$ defines an operator acting directly on $H$. To obtain an operator on the original Hilbert space, we instead define the contracted fourth-order moment operator directly by
\[
\mathsf K_{4,Y}^{\mathrm{ctr}}
:=
\mathbb E
\left[
\|Y\|_H^2
(Y\otimes Y)
\right].
\]
Since $\|Y\otimes Y\|_{\mathcal L_2(H)}=\|Y\|_H^2$, it follows that
$\bigl\|\|Y\|_H^2(Y\otimes Y)\bigr\|_{\mathcal L_2(H)}=\|Y\|_H^4$.
As Gaussian mixtures possess finite fourth moments, the above expectation is well defined as a Bochner integral with values in the Hilbert–Schmidt space.

Formally, this construction corresponds to contracting the first two tensor factors of the fourth-order tensor. Indeed, on elementary tensors,
\[
f_1\otimes f_2\otimes f_3\otimes f_4
\longmapsto
\langle f_1,f_2\rangle_H\,
(f_3\otimes f_4),
\]
which yields the representation
$
\|Y\|_H^2
(Y\otimes Y)
$
when applied to the elementary tensor
$
Y\otimes Y\otimes Y\otimes Y.
$
\bl{In infinite dimension, this elementary-tensor contraction does not extend to a bounded linear map from $H^{\otimes4}$ to $H^{\otimes2}$; accordingly, the preceding display is only a formal interpretation, while $\mathsf K_{4,Y}^{\mathrm{ctr}}$ is defined rigorously by the Bochner expectation above.}
Expanding the contraction with respect to the covariance-induced
Karhunen--Loève basis yields
\[
\mathsf K_{4,Y}^{\mathrm{ctr}}
=
\sum_{j,k\ge1}
K^{(4)}_{jk}
\,
\gamma_j\otimes\gamma_k,
\]
where
\[
K^{(4)}_{jk}
=
\sum_{\ell=1}^{2}
\pi_\ell
\,
\mathbb E
\!\left[
\|Y_\ell\|_H^2
(\xi_{\ell j}+a_\ell\nu_j)
(\xi_{\ell k}+a_\ell\nu_k)
\right].
\]

Since $\|Y_\ell\|_H^2=\sum_{r=1}^{\infty}(\xi_{\ell r}+a_\ell\nu_r)^2$, the series of expectations is absolutely summable by Cauchy--Schwarz together with the trace-class property of $\mathsf K$ and the assumption $m_\Delta\in H$. Hence Fubini's theorem for series justifies the interchange of the infinite sum and expectation. Thus,
\[
K_{jk}^{(4)}
=
\sum_{\ell=1}^{2}\pi_\ell
\sum_{r=1}^{\infty}
\mathbb E
\left[
(\xi_{\ell r}+a_\ell\nu_r)^2
(\xi_{\ell j}+a_\ell\nu_j)
(\xi_{\ell k}+a_\ell\nu_k)
\right].
\]
Expanding the product gives
\[
\begin{aligned}
&
\mathbb E
\left[
(\xi_{\ell r}+a_\ell\nu_r)^2
(\xi_{\ell j}+a_\ell\nu_j)
(\xi_{\ell k}+a_\ell\nu_k)
\right]
\\
&=
\mathbb E
(\xi_{\ell r}^2\xi_{\ell j}\xi_{\ell k})
+
a_\ell\nu_j
\mathbb E(\xi_{\ell r}^2\xi_{\ell k})
+
a_\ell\nu_k
\mathbb E(\xi_{\ell r}^2\xi_{\ell j})
\\
&\quad
+
a_\ell^2\nu_j\nu_k
\mathbb E(\xi_{\ell r}^2)
+
2a_\ell\nu_r
\mathbb E(\xi_{\ell r}\xi_{\ell j}\xi_{\ell k})
\\
&\quad
+
2a_\ell^2\nu_r\nu_j
\mathbb E(\xi_{\ell r}\xi_{\ell k})
+
2a_\ell^2\nu_r\nu_k
\mathbb E(\xi_{\ell r}\xi_{\ell j})
\\
&\quad
+
2a_\ell^3\nu_r\nu_j\nu_k
\mathbb E(\xi_{\ell r})
+
a_\ell^2\nu_r^2
\mathbb E(\xi_{\ell j}\xi_{\ell k})
\\
&\quad
+
a_\ell^3\nu_r^2\nu_j
\mathbb E(\xi_{\ell k})
+
a_\ell^3\nu_r^2\nu_k
\mathbb E(\xi_{\ell j})
+
a_\ell^4\nu_r^2\nu_j\nu_k.
\end{aligned}
\]
Since the Gaussian coefficients are centred, all terms containing an odd
number of Gaussian factors vanish. Therefore,
\[
\begin{aligned}
&
\mathbb E
\left[
(\xi_{\ell r}+a_\ell\nu_r)^2
(\xi_{\ell j}+a_\ell\nu_j)
(\xi_{\ell k}+a_\ell\nu_k)
\right]
\\
&=
\mathbb E
(\xi_{\ell r}^2\xi_{\ell j}\xi_{\ell k})
+
a_\ell^2\nu_j\nu_k
\mathbb E(\xi_{\ell r}^2)
\\
&\quad
+
2a_\ell^2\nu_r\nu_j
\mathbb E(\xi_{\ell r}\xi_{\ell k})
+
2a_\ell^2\nu_r\nu_k
\mathbb E(\xi_{\ell r}\xi_{\ell j})
\\
&\quad
+
a_\ell^2\nu_r^2
\mathbb E(\xi_{\ell j}\xi_{\ell k})
+
a_\ell^4\nu_r^2\nu_j\nu_k.
\end{aligned}
\]
Using
$
\mathbb E(\xi_{\ell j}\xi_{\ell k})
=
\lambda_j\delta_{jk}
$
and Isserlis' formula,
\[
\mathbb E
(\xi_{\ell r}^2\xi_{\ell j}\xi_{\ell k})
=
\lambda_r\lambda_j\delta_{jk}
+
2\lambda_r^2\delta_{rj}\delta_{rk},
\]
we obtain
\[
\begin{aligned}
K_{jk}^{(4)}
={}&
\left[
\operatorname{tr}(\mathsf K)\lambda_j
+
2\lambda_j^2
\right]\delta_{jk}
\\
&+
\left(
\sum_{\ell=1}^{2}\pi_\ell a_\ell^2
\right)
\Bigl[
\operatorname{tr}(\mathsf K)\nu_j\nu_k
+
2(\lambda_j+\lambda_k)\nu_j\nu_k
+
\|m_\Delta\|_H^2\lambda_j\delta_{jk}
\Bigr]
\\
&+
\left(
\sum_{\ell=1}^{2}\pi_\ell a_\ell^4
\right)
\|m_\Delta\|_H^2\nu_j\nu_k.
\end{aligned}
\]

Using \eqref{eq1} and \eqref{eq2}, these coefficients simplify to
\[
\begin{aligned}
K_{jk}^{(4)}
={}&
\left[
\operatorname{tr}(\mathsf K)\lambda_j
+
2\lambda_j^2
\right]\delta_{jk}
\\
&+
\rho
\Bigl[
\operatorname{tr}(\mathsf K)\nu_j\nu_k
+
2(\lambda_j+\lambda_k)\nu_j\nu_k
+
\|m_\Delta\|_H^2\lambda_j\delta_{jk}
\Bigr]
\\
&+
\rho(1-3\rho)
\|m_\Delta\|_H^2\nu_j\nu_k.
\end{aligned}
\]
Consequently, $\mathsf K_{4,Y}^{\mathrm{ctr}}=\sum_{j,k\ge1}K_{jk}^{(4)}(\gamma_j\otimes\gamma_k)$.
In general, the coefficients \(K_{jk}^{(4)}\) do not vanish when \(j\neq k\).  Hence, unlike the covariance operator, the contracted fourth-order moment operator is not diagonal with respect to the Karhunen–Loève basis. Consequently, its eigenfunctions are \bl{not, in general,} given by the covariance eigenfunctions, and its spectral decomposition cannot be obtained directly from the coefficient expansion above. We therefore base the subsequent analysis on the covariance-based fourth-order representation $\mathsf V^{(4)}$ introduced above. After standardization, its \bl{completely diagonal structure} admits an explicit characterization, which, as shown in the next section, is naturally expressed in terms of the Cameron–Martin geometry.

\section{Main results}

\subsection{Standardization and contracted fourth-order moments}
As discussed in the previous section, the infinite-dimensional Fisher
discriminant is defined only under a stronger assumption than the
Gaussian measure equivalence criterion. Throughout the remainder of the
paper, we therefore work with the finite-dimensional Fisher
discriminants obtained from Karhunen--Loève truncations.

The finite-dimensional truncation also provides a well-defined
standardization of the Gaussian random element. In infinite-dimensional
Hilbert spaces, the covariance operator $\mathsf K$ is compact, so that
$\mathsf K^{-1/2}$ is an unbounded operator. Moreover, a Gaussian random
element belongs to the Cameron--Martin space with probability zero
\cite{Lukic01}. Consequently,
$
\mathsf K^{-1/2}(X_\ell-m_\ell)
$
does not define an $H$-valued random element. Instead, standardization
is performed on each finite-dimensional Karhunen--Loève subspace, where
the projected covariance operator is invertible. Unlike regularization
methods that modify the underlying metric, spectral truncation preserves
the covariance geometry on each finite-dimensional subspace while
allowing the infinite-dimensional behaviour to emerge as the truncation
level increases.

Let
$
H_q=\operatorname{span}\{\gamma_1,\ldots,\gamma_q\},
$
and let
$
\Pi_q
$
denote the orthogonal projection onto $H_q$. By the assumption that $\lambda_j>0$ for every $j$, the projected covariance operator
$
\mathsf K_q=\Pi_q\mathsf K\Pi_q
$
is positive definite on $H_q$, so that
$
\mathsf K_q^{-1/2}
$
is well defined. 

Writing
\[
m_{\Delta,q}=\Pi_qm_\Delta,
\qquad
\widetilde m_{\Delta,q}
=
\mathsf K_q^{-1/2}m_{\Delta,q}
=
\sum_{j=1}^q
\widetilde\nu_j\gamma_j,
\qquad
\widetilde\nu_j
=
\nu_j/\sqrt{\lambda_j}.
\]
define, for each class $\ell=1,2$, the standardized centred Gaussian random element
\[
Z_{q,\ell}
=
\mathsf K_q^{-1/2}\Pi_q(X_\ell-m_\ell)
=
\sum_{j=1}^{q}\zeta_{\ell j}\gamma_j,
\]
where
\(
\zeta_{\ell j}=\xi_{\ell j}/\sqrt{\lambda_j}
\)
are independent standard Gaussian random variables. Since
\[
\mathsf K_q^{-1/2}\mathsf K_q\mathsf K_q^{-1/2}
=
\sum_{j=1}^{q}\gamma_j\otimes\gamma_j
=
\mathsf I_{H_q},
\]
it follows that $Z_{q,\ell}\sim\mathbb G(0,\mathsf I_{H_q})$, $\ell=1,2$, where $\mathsf I_{H_q}$ denotes the identity operator on $H_q$.
Define
\[
\widetilde Y_{q,\ell}
=
Z_{q,\ell}
+
a_\ell\widetilde m_{\Delta,q},
\qquad
\ell=1,2.
\]
Thus,
\[
\widetilde Y_q=
\begin{cases}
\widetilde Y_{q,1},
&\text{with probability }\alpha,\\
\widetilde Y_{q,2},
&\text{with probability }1-\alpha.
\end{cases}
\]

The finite-dimensional representation preserves the covariance-induced geometry introduced in the previous section while allowing the second- and fourth-order moment operators to be studied through explicit spectral calculations. The following proposition shows that, after standardization, both the covariance operator and the contracted fourth-order moment operator reduce to rank-one perturbations of the identity.

\begin{lemma} \label{lemma1}
Let $\mathsf K_{2,q}$ and $\mathsf K_{4,q}^{\mathrm{ctr}}$ denote,
respectively, the covariance operator and the contracted fourth-order
moment operator of the standardized mixture.
Then
\[
\mathsf K_{2,q}
=
\mathsf I_{H_q}
+
\rho\,
\widetilde m_{\Delta,q}\otimes
\widetilde m_{\Delta,q},
\]
and
\[
\mathsf K_{4,q}^{\mathrm{ctr}}
=
c_{0,q}\mathsf I_{H_q}
+
c_{1,q}\,
\widetilde m_{\Delta,q}\otimes
\widetilde m_{\Delta,q},
\]
where
\[
 c_{0,q}=q+2+\rho\|\widetilde m_{\Delta,q}\|_{H_q}^{2},
 \qquad
 c_{1,q}=\rho(q+4)+\rho(1-3\rho)\|\widetilde m_{\Delta,q}\|_{H_q}^{2}.
\]
\end{lemma}

Lemma~\ref{lemma1} shows that, after standardization, the contracted fourth-order moment operator retains exactly the same rank-one perturbation structure as the covariance operator. Consequently, its spectral decomposition is completely determined by the covariance geometry and does not isolate the coordinatewise fourth-order contributions that are the object of the subsequent analysis. 
Lemma~\ref{lemma1} also shows that the standardized discriminative direction $\widetilde m_{\Delta,q}$ is an eigenfunction of $\mathsf K_{4,q}^{\mathrm{ctr}}$ \cite{Pena14,Vidal25moment}. 
Since $\|\widetilde m_{\Delta,q}\|_{H_q}^{2}
=\|m_{\Delta,q}\|_{H(\mathfrak K)}^{2}$, we write
\[
r_q
:=
\|m_{\Delta,q}\|_{H(\mathfrak K)}^{2}
=
\|\widetilde m_{\Delta,q}\|_{H_q}^{2}.
\]
The corresponding discriminative eigenvalue equals
$\lambda_{\Delta,q}^{\mathrm{ctr}}=c_{0,q}+c_{1,q}r_q$,
whereas every direction orthogonal to $\widetilde m_{\Delta,q}$ has eigenvalue $c_{0,q}$.

Thus, the spectral contrast is entirely governed by the Cameron--Martin energy $r_q$. Nevertheless, because both $c_{0,q}$ and $c_{1,q}$ depend on the ambient dimension $q$, this operator does not isolate coordinatewise fourth-order contributions. This motivates the spectral analysis of the fourth-order component $\mathsf V^{(4)}$ introduced in the previous section.

\subsection{Spectral representations}
For each class $\ell$, define the random tensor
\[
\widetilde{\mathsf V}^{(2)}_{\ell,q}
=
\sum_{j,k=1}^{q}
\left(
\zeta_{\ell j}\zeta_{\ell k}
+
a_\ell^2\widetilde\nu_j\widetilde\nu_k
\right)
(\gamma_j\otimes\gamma_k).
\]

The standardized counterpart of $\mathsf V^{(4)}$ is
\[
\widetilde{ \mathsf V}_{q}^{(4)}
=
\sum_{\ell=1}^{2}
\pi_\ell
\mathbb E
\left(
\widetilde{\mathsf V}^{(2)}_{\ell,q}
\otimes
\widetilde{\mathsf V}^{(2)}_{\ell,q}
\right).
\]

Expanding with respect to the orthonormal basis
$\{\gamma_j\otimes\gamma_k\}_{j,k=1}^{q}$ gives
\[
\widetilde{ \mathsf V}_{q}^{(4)}
=
\sum_{j,k,m,n=1}^{q}
\widetilde{V}_{jkmn}^{(q)}
\gamma_j\otimes\gamma_k
\otimes
\gamma_m\otimes\gamma_n.
\]
where
\[
\begin{aligned}
\widetilde{V}_{jkmn}^{(q)}
=
\sum_{\ell=1}^{2}\pi_\ell
\Bigl\{&
\mathbb E
(\zeta_{\ell j}\zeta_{\ell k}\zeta_{\ell m}\zeta_{\ell n})
\\
&+
a_\ell^2
\widetilde\nu_j\widetilde\nu_k
\,
\mathbb E(\zeta_{\ell m}\zeta_{\ell n})
+
a_\ell^2
\widetilde\nu_m\widetilde\nu_n
\,
\mathbb E(\zeta_{\ell j}\zeta_{\ell k})
\\
&+
a_\ell^4
\widetilde\nu_j\widetilde\nu_k
\widetilde\nu_m\widetilde\nu_n
\Bigr\}.
\end{aligned}
\]
Since the standardized Gaussian coefficients satisfy
$
\mathbb E(\zeta_{\ell j}\zeta_{\ell k})=\delta_{jk},
$
Isserlis' formula yields
\[
\begin{aligned}
\widetilde{V}_{jkmn}^{(q)}
={}&
\delta_{jk}\delta_{mn}
+
\delta_{jm}\delta_{kn}
+
\delta_{jn}\delta_{km}
\\
&+
\rho
\left[
\widetilde\nu_j\widetilde\nu_k\delta_{mn}
+
\widetilde\nu_m\widetilde\nu_n\delta_{jk}
\right]
\\
&+
\rho(1-3\rho)
\widetilde\nu_j\widetilde\nu_k
\widetilde\nu_m\widetilde\nu_n.
\end{aligned}
\]

The coefficients $\widetilde V_{jkmn}^{(q)}$ represent the matrix entries of
$\widetilde{\mathsf V}_q^{(4)}$ with respect to the tensor-product
Karhunen--Loève basis.
Although the covariance operator is diagonal in this basis, the associated
fourth-order operator is not. The Gaussian pairings arising from Isserlis'
formula couple distinct tensor directions, while the mean-shift contribution
introduces additional interactions through the coefficients
$\widetilde\nu_j$. For example, discarding the mean-shift terms, the Gaussian part yields
$\widetilde V_{jkmn}^{(q)}=0$ whenever no pairing of the indices is possible. Moreover, $\widetilde V_{jkjk}^{(q)}=1$ and $\widetilde V_{kjjk}^{(q)}=1$, $j\neq k$, coupling the tensor directions
$\gamma_j\otimes\gamma_k$
and
$\gamma_k\otimes\gamma_j$.
Likewise, $\widetilde V_{jjkk}^{(q)}=1$, $j\neq k$, couples the diagonal tensor directions
$\gamma_j\otimes\gamma_j$
and
$\gamma_k\otimes\gamma_k$,
whereas $\widetilde V_{jjjj}^{(q)}=3$.

The preceding discussion shows that the operator associated with \bl{$\widetilde{\mathsf V}_q^{(4)}$} couples
different tensor directions, so that its spectrum cannot be characterized
directly through the tensor-product Karhunen--Loève basis. We therefore
consider the diagonal coordinates of the random tensor
$\widetilde{\mathsf V}^{(2)}_{\ell,q}$. The coordinate along the diagonal tensor direction
$\gamma_j\otimes\gamma_j$ is
\[
Q_{\ell j}
:=
\left\langle
\widetilde{\mathsf V}^{(2)}_{\ell,q},
\gamma_j\otimes\gamma_j
\right\rangle_{\mathcal L_2(H_q)}
=
\zeta_{\ell j}^{\,2}
+
a_\ell^{2}\widetilde\nu_j^{\,2},
\qquad
j=1,\ldots,q.
\]

\begin{lemma} \label{lem:diagonal-coefficients}
For every $j=1,\ldots,q$, the completely diagonal coefficient of
$\widetilde{\mathsf V}_q^{(4)}$ satisfies
\[
\kappa_{j,q}
:=
\widetilde V_{jjjj}^{(q)}
=
\sum_{\ell=1}^{2}
\pi_\ell
\mathbb E\bigl(Q_{\ell j}^{2}\bigr)
=
3
+
2\rho\,\widetilde\nu_j^{\,2}
+
\rho(1-3\rho)\widetilde\nu_j^{\,4}.
\]
\end{lemma}

The lemma above shows that the completely diagonal coefficients of
$\widetilde{\mathsf V}_q^{(4)}$ are precisely the second moments of the diagonal
coordinates of $\widetilde{\mathsf V}^{(2)}_{\ell,q}$. Thus, each Karhunen–Loève direction 
is naturally associated with a scalar fourth-order quantity through the corresponding 
diagonal coordinate of  $\widetilde{\mathsf V}^{(2)}_{\ell,q}$.

Under this correspondence, identifying
$\gamma_j\otimes\gamma_j$
with the corresponding rank-one projection on $H_q$, define the operator
\[
\mathsf{\widetilde K}_{4,q}
=
\sum_{j=1}^{q}
\kappa_{j,q}
(\gamma_j\otimes\gamma_j).
\]
By construction,
$\mathsf{\widetilde K}_{4,q}\gamma_j=\kappa_{j,q}\gamma_j$,
$j=1,\ldots,q$, so the Karhunen--Loève eigenfunctions remain eigenfunctions of
$\mathsf{\widetilde K}_{4,q}$, with corresponding eigenvalues
$\kappa_{j,q}$.

The operator $\mathsf{\widetilde K}_{4,q}$ is obtained from
$\widetilde{\mathsf V}_q^{(4)}$ by extracting its completely diagonal entries in the tensor-product Karhunen–Loève basis. It is therefore neither the restriction nor the compression of $\widetilde{\mathsf V}_q^{(4)}$ to the diagonal tensor subspace, since the latter also contains the non-zero coefficients $\widetilde V_{jjkk}^{(q)}$, $j\neq k$. Rather, it isolates the marginal fourth-order contribution associated with each covariance eigenfunction. Equivalently, 
we have
\[
\mathsf{\widetilde K}_{4,q}
=
\sum_{j=1}^{q}
\left\langle
\widetilde{\mathsf V}_q^{(4)}
(\gamma_j\otimes\gamma_j),
\gamma_j\otimes\gamma_j
\right\rangle_{\mathcal L_2(H_q)}
(\gamma_j\otimes\gamma_j).
\]
showing that $\mathsf{\widetilde K}_{4,q}$ is obtained by a linear extraction of the completely diagonal coefficients of $\widetilde{\mathsf V}_q^{(4)}$.

The completely diagonal coefficients admit a natural geometric interpretation in
terms of the Cameron--Martin geometry induced by the covariance operator.
The Cameron--Martin space
associated with $\mathsf K$ is
\[
H(\mathfrak K)
=
\left\{
f=\sum_{j=1}^{\infty}f_j\gamma_j\in H:
\sum_{j=1}^{\infty}\frac{f_j^2}{\lambda_j}<\infty
\right\},
\]
equipped with the inner product
\[
\langle f,g\rangle_{H(\mathfrak K)}
=
\sum_{j=1}^{\infty}
\frac{
\langle f,\gamma_j\rangle_H
\langle g,\gamma_j\rangle_H
}{\lambda_j}.
\]
In particular,
\[
\|m_\Delta\|_{H(\mathfrak K)}^2
=
\sum_{j=1}^{\infty}
\frac{\nu_j^2}{\lambda_j}.
\]

Analogously, the operator $\mathsf K^2$ has eigenfunctions
$\{\gamma_j\}_{j\ge1}$ and eigenvalues $\{\lambda_j^2\}_{j\ge1}$.
The associated Hilbert space is
\[
H_2(\mathfrak K)
=
\left\{
f=\sum_{j=1}^{\infty}f_j\gamma_j\in H:
\sum_{j=1}^{\infty}\frac{f_j^2}{\lambda_j^2}<\infty
\right\},
\]
with inner product
\[
\langle f,g\rangle_{H_2(\mathfrak K)}
=
\sum_{j=1}^{\infty}
\frac{
\langle f,\gamma_j\rangle_H
\langle g,\gamma_j\rangle_H
}{\lambda_j^2}.
\]

For each Karhunen--Loève direction, define
\[
m_{\Delta,j}
=
\nu_j\gamma_j,
\qquad
m_{\Delta,j}^{\langle 2 \rangle}
=
\nu_j^2\gamma_j.
\]
Since these elements are supported on a single covariance eigenfunction,
their squared norms reduce to
\[
\|m_{\Delta,j}\|_{H(\mathfrak K)}^2
=
\frac{\nu_j^2}{\lambda_j}
=
\widetilde\nu_j^{\,2}, \quad 
\|m_{\Delta,j}^{\langle 2 \rangle}\|_{H_2(\mathfrak K)}^2
=
\frac{\nu_j^4}{\lambda_j^2}
=
\widetilde\nu_j^{\,4}.
\]
Substituting these identities into the expression for
$\kappa_{j,q}$ yields
\[
\kappa_{j,q}
=
3
+
2\rho
\|m_{\Delta,j}\|_{H(\mathfrak K)}^2
+
\rho(1-3\rho)
\|m_{\Delta,j}^{\langle 2 \rangle}\|_{H_2(\mathfrak K)}^2.
\]

Thus, each diagonal fourth-order coefficient combines the coordinatewise Cameron--Martin energy
of the mean difference with the corresponding quadratic norm induced by
$\mathsf K^2$. Summing over all coordinates yields the global energies
\[
\|m_{\Delta,q}\|_{H(\mathfrak K)}^2
=
\sum_{j=1}^{q}\frac{\nu_j^2}{\lambda_j},
\qquad
\left\|
\sum_{j=1}^{q}m_{\Delta,j}^{\langle 2 \rangle}
\right\|_{H_2(\mathfrak K)}^2
=
\sum_{j=1}^{q}\frac{\nu_j^4}{\lambda_j^2},
\]
which will determine the asymptotic behaviour of the diagonal fourth-order
spectrum as $q\to\infty$.
These cumulative quantities will later define the aggregated fourth-order spectral functionals.

The standardized fourth-order tensor $\widetilde{\mathsf V}_q^{(4)}$ induces a coordinatewise operator on $H_q$ by expressing the tensor in a fixed Karhunen--Loève basis and retaining only its completely diagonal coefficients\footnote{This construction differs from the kurtosis operator introduced by Pe\~na, Prieto and Rend\'on \cite{Pena14}. Their operator is intrinsically defined and the discriminant eigendirection follows from the rotational invariance of the standardized Gaussian mixture. By contrast, the present operator is the coordinatewise representation obtained by extracting the marginal fourth-order contributions associated with each Karhunen--Loève direction.}. The resulting operator is therefore not defined independently of the Karhunen--Loève representation. When the covariance operator has repeated eigenvalues, different orthonormal bases within the corresponding eigenspaces induce different coordinatewise operators, although they represent the same underlying fourth-order tensor. Accordingly, all subsequent spectral statements are understood relative to the fixed Karhunen--Loève representation. \bl{The eigenfunction property associated with this construction had previously been asserted without a precise characterization of when it holds \cite{Vidal25moment}. The} following result \bl{establishes the exact necessary and sufficient condition under which} the truncated Fisher discriminant belongs to one of its eigenspaces.

\begin{theorem} \label{T1}
Assume that
$
m_{\Delta,q}\neq0
$
and let
$
\beta_q
=
\mathsf K_q^{\dagger}m_{\Delta,q}
=
\mathsf K_q^{\dagger/2}\widetilde m_{\Delta,q}
$
denote the truncated Fisher discriminant. Then
$
\beta_q
$
is an eigenfunction of the induced coordinatewise operator
$
\mathsf{\widetilde K}_{4,q}
$
if and only if there exists a constant
$
c>0
$
such that
$
\widetilde\nu_j^{\,2}=c
$
for every
$
j
$
with
$
\widetilde\nu_j\neq0.
$
In that case,
\[
\mathsf{\widetilde K}_{4,q}\beta_q=\kappa_{\Delta,q}\beta_q,
\qquad
\kappa_{\Delta,q}=3+2\rho c+\rho(1-3\rho)c^2.
\]
In particular, the condition is automatically satisfied whenever
$
\widetilde m_{\Delta,q}
$
is supported on a single Karhunen--Loève direction.
\end{theorem}

\begin{remark}
The condition of Theorem~\ref{T1} does not require the discriminant to
be supported on a single Karhunen--Loève direction. Several active
directions may coexist, provided they share the same standardized
discriminative energy $\widetilde\nu_j^{\,2}=c$. In that case,
$
E_c
:=
\operatorname{span}
\left\{
\gamma_j:
\widetilde\nu_j\neq0
\right\}
$
is contained in a single eigenspace of $\mathsf{\widetilde K}_{4,q}$, and
$\beta_q$ is the particular element of $E_c$ determined by the coefficients
$\widetilde\nu_j/\sqrt{\lambda_j}$.
\end{remark}

\subsection{Aggregated spectral functionals and Gaussian measure singularity}
The fourth-order tensor $\widetilde{\mathsf V}^{(4)}_q$ introduced above is defined as the 
second moment of the random tensor $\widetilde{\mathsf V}^{(2)}_{\ell,q}$. A natural deterministic counterpart
is obtained by replacing the second moment within each mixture component by
the tensor product of its first moment. Accordingly, define
\[
\mathsf F_q
:=
\sum_{\ell=1}^{2}
\pi_\ell
\bigl(
\mathbb E\widetilde{\mathsf V}^{(2)}_{\ell,q}
\bigr)
\otimes
\bigl(
\mathbb E\widetilde{\mathsf V}^{(2)}_{\ell,q}
\bigr).
\]

Since
\[
\begin{aligned}
\mathbb E\widetilde{\mathsf V}^{(2)}_{\ell,q}
&=
\sum_{j,k=1}^{q}
\left(
\delta_{jk}
+
a_\ell^2\widetilde\nu_j\widetilde\nu_k
\right)
(\gamma_j\otimes\gamma_k) \\
&=
\mathsf I_{H_q}
+
a_\ell^2
\widetilde m_{\Delta,q}\otimes
\widetilde m_{\Delta,q},
\end{aligned}
\]
the completely diagonal coefficients of $\mathsf F_q$ are
\[
\kappa_{j,q}^{(F)}
=
1
+
2\rho\,\widetilde\nu_j^{\,2}
+
\rho(1-3\rho)\widetilde\nu_j^{\,4}.
\]
As before, these coefficients induce the operator
\[
\mathsf{\widetilde K}_q^{(F)}
=
\sum_{j=1}^{q}
\kappa_{j,q}^{(F)}
(\gamma_j\otimes\gamma_j),
\]
whose eigenfunctions are the Karhunen--Loève eigenfunctions.
As with $\mathsf{\widetilde K}_{4,q}$, the operator
$\mathsf{\widetilde K}_q^{(F)}$ retains only the completely diagonal
coefficients of $\mathsf F_q$ and is not the compression of $\mathsf F_q$
to the diagonal tensor subspace.

The induced operator
$\mathsf{\widetilde K}_q^{(F)}$
is the deterministic counterpart of
$\mathsf{\widetilde K}_{4,q}$.
Their corresponding eigenvalues satisfy
$\kappa_{j,q}=\kappa_{j,q}^{(F)}+2$, where the constant $2$ represents the purely Gaussian fourth-order baseline, since $\mathbb E(\zeta^4)=3$, $\bigl(\mathbb E(\zeta^2)\bigr)^2=1$, and $\zeta\sim N(0,1)$.
Thus, the discriminative contributions of both constructions coincide exactly, and they differ only through this nondiscriminative Gaussian component.

Since each diagonal coefficient contains a constant Gaussian contribution
($1$ for $\mathsf{\widetilde K}_{q}^{(F)}$ and $3$ for
$\mathsf{\widetilde K}_{4,q}$), a direct summation of the eigenvalues would
produce a quantity growing linearly with $q$ even in the absence of any
mean difference. To isolate the discriminative contribution, the Gaussian
baseline is removed before aggregation and restored only once as a reference
level. This yields the aggregated spectral functionals
\[
\mathcal A_q^{(F)}
:=
1+
\sum_{j=1}^{q}
\bigl(
\kappa_{j,q}^{(F)}-1
\bigr),
\quad
\mathcal A_q^{(4)}
:=
3+
\sum_{j=1}^{q}
\bigl(
\kappa_{j,q}-3
\bigr).
\]
Substituting the previous expressions yields
\[
\mathcal A_q^{(F)}
=
1+
2\rho
\sum_{j=1}^{q}
\widetilde\nu_j^{\,2}
+
\rho(1-3\rho)
\sum_{j=1}^{q}
\widetilde\nu_j^{\,4},
\]
and
\[
\mathcal A_q^{(4)}
=
3+
2\rho
\sum_{j=1}^{q}
\widetilde\nu_j^{\,2}
+
\rho(1-3\rho)
\sum_{j=1}^{q}
\widetilde\nu_j^{\,4}.
\]
Hence $\mathcal A_q^{(4)}=\mathcal A_q^{(F)}+2$.

\bl{This exact identity depends on having restricted attention to $\mathsf V^{(4)}$ rather than the complete fourth-order moment tensor $\mathsf M_Y^{(4)}$ introduced in Section~2. It is natural to ask whether the excluded interaction terms affect this picture, and the answer can be computed exactly rather than merely bounded. This comparison serves as a benchmark, separating properties that hold for fourth-order moment tensors in general from those that are specific to the structured construction considered here.}

\bl{For each class $\ell$, the completely diagonal coefficient of the full standardized fourth-order tensor at direction $(j,j,j,j)$ is simply the fourth moment of the single coordinate $\zeta_{\ell j}+a_\ell\widetilde\nu_j$. Since $\mathbb E\zeta=\mathbb E\zeta^3=0$ for $\zeta\sim N(0,1)$, mixing over $\ell$ via \eqref{eq1}--\eqref{eq2} gives}
\[
\bl{
\kappa_{j,q}^{\mathrm{full}}
:=
3+6\rho\,\widetilde\nu_j^{\,2}+\rho(1-3\rho)\widetilde\nu_j^{\,4}.
}
\]
\bl{Comparing with Lemma~\ref{lem:diagonal-coefficients},}
\[
\bl{
\kappa_{j,q}^{\mathrm{full}}-\kappa_{j,q}=4\rho\,\widetilde\nu_j^{\,2}.
}
\]
\bl{This difference is not hidden: it is exactly the diagonal contribution of the interaction term $\mathsf R_\ell$ that was excluded from $\mathsf V^{(4)}$ in Section~2. Evaluating $\sum_{\ell=1}^2\pi_\ell\,\mathbb E(\mathsf R_\ell\otimes\mathsf R_\ell)$ at the diagonal direction $(j,j,j,j)$, under standardization, gives $\sum_{\ell=1}^2\pi_\ell\cdot4a_\ell^2\widetilde\nu_j^{\,2}=4\rho\widetilde\nu_j^{\,2}$ by \eqref{eq1}, while $\sum_{\ell=1}^2\pi_\ell\,\mathbb E(\mathsf V_\ell^{(2)}\otimes\mathsf R_\ell)$ and $\sum_{\ell=1}^2\pi_\ell\,\mathbb E(\mathsf R_\ell\otimes\mathsf V_\ell^{(2)})$ vanish identically, since every entry involves an odd number of centred Gaussian factors.}

\bl{Summing as before yields the full-tensor diagonal aggregate}
\[
\bl{
\mathcal A_q^{\mathrm{full}}
:=
3+\sum_{j=1}^{q}\bigl(\kappa_{j,q}^{\mathrm{full}}-3\bigr)
=
3+6\rho\sum_{j=1}^{q}\widetilde\nu_j^{\,2}+\rho(1-3\rho)\sum_{j=1}^{q}\widetilde\nu_j^{\,4},
}
\]
\bl{and comparison with $\mathcal A_q^{(4)}$ gives the exact relation}
\[
\bl{
\mathcal A_q^{\mathrm{full}}=\mathcal A_q^{(4)}+4\rho\,r_q,
\qquad
r_q=\sum_{j=1}^{q}\widetilde\nu_j^{\,2}.
}
\]

\bl{\begin{proposition}\label{prop:full-tensor}
Let $P_1=\mathbb G(m_1,\mathsf K)$ and $P_2=\mathbb G(m_2,\mathsf K)$ denote the Gaussian measures introduced at the outset, with common covariance operator $\mathsf K$ satisfying $\lambda_j>0$ for every $j$. Then $P_1\perp P_2$ if and only if $\mathcal A_q^{\mathrm{full}}\to\infty$.
\end{proposition}}

\bl{Thus the completely diagonal aggregate of the full fourth-order moment tensor also detects Gaussian measure singularity through divergence. What is special to the pairing of $\mathsf V^{(4)}$ with its product-of-expectations counterpart $\mathsf F_q$ is not that it alone can detect singularity, but the exact identity $\mathcal A_q^{(4)}-\mathcal A_q^{(F)}=2$: because this gap is a fixed constant, independent of $q$ and of how the discriminative energy is distributed across coordinates, the ratio $\mathcal A_q^{(4)}/\mathcal A_q^{(F)}$ converges to $1$ universally in the singular regime. The complete fourth-order tensor does not share this property, as the following remark makes precise.}

\begin{remark}
\bl{
Writing $t_q:=\sum_{j=1}^{q}\widetilde\nu_j^{\,4}$, one has, for every $q$,
\[
1<\frac{\mathcal A_q^{\mathrm{full}}}{\mathcal A_q^{(F)}}\le3,
\]
since $\mathcal A_q^{\mathrm{full}}/\mathcal A_q^{(F)}-1=2(1+2\rho r_q)/\mathcal A_q^{(F)}>0$ and $3-\mathcal A_q^{\mathrm{full}}/\mathcal A_q^{(F)}=2\rho(1-3\rho)t_q/\mathcal A_q^{(F)}\ge0$. If, in the singular regime, $t_q/r_q\to\tau\in[0,\infty)$ as $q\to\infty$, then
\[
\frac{\mathcal A_q^{\mathrm{full}}}{\mathcal A_q^{(F)}}\longrightarrow\frac{6\rho+\rho(1-3\rho)\tau}{2\rho+\rho(1-3\rho)\tau}\in(1,3],
\]
If instead $t_q/r_q\to\infty$, the limiting ratio equals $1$. Unlike $\mathcal A_q^{(4)}/\mathcal A_q^{(F)}$, which tends to $1$ regardless of how the discriminative energy is distributed across coordinates, the limiting ratio for the complete fourth-order tensor depends on the profile of the mean shift through $\tau$.
}
\end{remark}

Now, define the common discriminative contribution
\begin{equation}\label{discener}
\Delta_q
:=
2\rho\sum_{j=1}^{q}\widetilde\nu_j^{\,2}
+
\rho(1-3\rho)\sum_{j=1}^{q}\widetilde\nu_j^{\,4}.
\end{equation}
Then $\mathcal A_q^{(4)}=3+\Delta_q$ and
$\mathcal A_q^{(F)}=1+\Delta_q$.
The asymptotic relationship between these aggregated spectral functionals
is established in the following theorem. 

\begin{theorem}
\label{thm:asymptotic-equivalence}
Suppose that
$
m_\Delta\notin H(\mathfrak K).
$
Then
\[
\frac{\mathcal A_q^{(4)}}{\mathcal A_q^{(F)}}\longrightarrow1,
\qquad q\to\infty,
\qquad\text{equivalently,}\qquad
\mathcal A_q^{(4)}\sim\mathcal A_q^{(F)}.
\]
\end{theorem}

The asymptotic equivalence reflects the fact that the two constructions
differ only through the Gaussian baseline contribution. Indeed,
$
\mathcal A_q^{(4)}-\mathcal A_q^{(F)}=2,
$
whereas
$
\mathcal A_q^{(F)}\to\infty
$
in the singular regime. Consequently,
\[
\frac{
\mathcal A_q^{(4)}
-
\mathcal A_q^{(F)}
}{
\mathcal A_q^{(F)}
}
=
\frac{2}{1+\Delta_q}
\longrightarrow0,
\]
showing that the constant Gaussian contribution becomes asymptotically negligible relative to the diverging discriminative component. 

This equivalence does not introduce a new boundary between equivalence and
singularity: that boundary remains the classical Cameron--Martin condition.
The fourth-order contribution modifies the magnitude of the aggregated
spectral quantities, but not the criterion determining whether they diverge.

Although
$\mathsf K_{4,q}^{\mathrm{ctr}}$
and
$\mathsf{\widetilde K}_{4,q}$
arise from different fourth-order constructions, both contain the
Cameron--Martin energy
$
r_q
$
as the quantity governing their asymptotic discriminative growth.
The discriminative eigenvalue of
$\mathsf K_{4,q}^{\mathrm{ctr}}$
depends on $r_q$ through Lemma~\ref{lemma1}, whereas the aggregated diagonal functional
$\mathcal A_q^{(4)}$
depends on \eqref{discener}.
Thus, the two constructions differ through additional finite-dimensional
contributions, but these do not alter the divergence criterion, which is
determined entirely by the Cameron--Martin energy.

The next theorem establishes that the divergence of the aggregated
discriminative contribution is equivalent to the singularity of the
underlying Gaussian measures in the sense of Feldman--Hájek. Recall the
definition of $\Delta_q$ in \eqref{discener}.

\begin{theorem}%[Spectral characterization of the singular regime]
\label{T2}
Fix $\alpha\in(0,1)$ and let $\rho=\alpha(1-\alpha)$. Let $P_1=\mathbb G(m_1,\mathsf K)$ and $P_2=\mathbb G(m_2,\mathsf K)$ be Gaussian measures on a separable Hilbert space $H$ with common covariance
operator $\mathsf K$, whose eigenvalues satisfy $\lambda_j>0$ for every $j$,
and let $m_\Delta=m_2-m_1$.
Then the following statements are equivalent:
\[
P_1\perp P_2,
\qquad
\Delta_q\longrightarrow\infty,
\qquad
\frac{\mathcal A_q^{(4)}}{\mathcal A_q^{(F)}}\longrightarrow1.
\]
Equivalently, $P_1\perp P_2\iff\mathcal A_q^{(4)}\sim\mathcal A_q^{(F)}$.
\end{theorem}

The previous theorem shows that the convergence $\mathcal A_q^{(4)}/\mathcal A_q^{(F)}\longrightarrow1$ is equivalent, under the common-covariance model, to mutual singularity of
the underlying Gaussian measures. The equivalence--singularity boundary
itself remains the classical Cameron--Martin condition; the contribution of
the present construction is to show how this condition is encoded by
aggregated quantities derived from the completely diagonal coefficients
\bl{of $\mathsf V^{(4)}$. Proposition~\ref{prop:full-tensor} shows that the complete fourth-order moment tensor encodes the same singularity criterion through divergence of its own diagonal aggregate, although not through a ratio converging universally to $1$.}

In the general heteroscedastic setting, equivalence additionally requires
the relative covariance operator to be a Hilbert--Schmidt perturbation of
the identity, giving rise to the regularized Fredholm determinants appearing
in the explicit Radon--Nikodym derivatives of Shepp \cite{SheppRN66}; see
also \cite{Hediger26} for a recent RKHS treatment of this covariance-driven
equivalence problem. In the homoscedastic framework considered here, this
covariance contribution is trivial, and the dichotomy is governed entirely
by the Cameron--Martin norm of the mean difference.

\appendix
\section{Proofs of formal statements}

\subsection{Proof of lemma~\ref{lemma1}}
\begin{proof}
Write
\[
b_\ell=a_\ell\widetilde m_{\Delta,q},
\qquad
\pi_1=\alpha,\qquad
\pi_2=1-\alpha.
\]
Conditionally on class $\ell$,
$
\widetilde Y_{q,\ell}=Z_{q,\ell}+b_\ell,
$
where
\[
\mathbb E(Z_{q,\ell})=0,
\qquad
\mathbb E(Z_{q,\ell}\otimes Z_{q,\ell})=\mathsf I_{H_q}.
\]
Hence
\[
\begin{aligned}
\mathsf K_{2,q}
&=\sum_{\ell=1}^{2}\pi_\ell
\mathbb E\!\left[(Z_{q,\ell}+b_\ell)\otimes(Z_{q,\ell}+b_\ell)\right] \\
&=\mathsf I_{H_q}
+\sum_{\ell=1}^{2}\pi_\ell(b_\ell\otimes b_\ell) \\
&=\mathsf I_{H_q}
+\rho\,\widetilde m_{\Delta,q}\otimes\widetilde m_{\Delta,q}.
\end{aligned}
\]

Applying the coefficient expansion for the contracted fourth-order moment operator derived in the previous section, under standardization
\[
\lambda_j=1,
\qquad
j=1,\ldots,q,
\qquad
\operatorname{tr}(\mathsf I_{H_q})=q,
\]
we obtain
\[
\begin{aligned}
K_{jk}^{(4)}
={}&
(q+2)\delta_{jk} \\
&+
\left(
\sum_{\ell=1}^{2}\pi_\ell a_\ell^2
\right)
\Bigl[
(q+4)\widetilde\nu_j\widetilde\nu_k
+
\|\widetilde m_{\Delta,q}\|_{H_q}^{2}\delta_{jk}
\Bigr] \\
&+
\left(
\sum_{\ell=1}^{2}\pi_\ell a_\ell^4
\right)
\|\widetilde m_{\Delta,q}\|_{H_q}^{2}
\widetilde\nu_j\widetilde\nu_k.
\end{aligned}
\]

Using \eqref{eq1} and \eqref{eq2},
\[
\begin{aligned}
K_{jk}^{(4)}
={}&
\left[
q+2+\rho\|\widetilde m_{\Delta,q}\|_{H_q}^{2}
\right]\delta_{jk} \\
&+
\left[
\rho(q+4)
+\rho(1-3\rho)\|\widetilde m_{\Delta,q}\|_{H_q}^{2}
\right]
\widetilde\nu_j\widetilde\nu_k.
\end{aligned}
\]

Since
\[
\mathsf I_{H_q}
=
\sum_{j=1}^{q}\gamma_j\otimes\gamma_j,
\qquad
\widetilde m_{\Delta,q}\otimes\widetilde m_{\Delta,q}
=
\sum_{j,k=1}^{q}
\widetilde\nu_j\widetilde\nu_k
(\gamma_j\otimes\gamma_k),
\]
it follows that
\[
\mathsf K_{4,q}^{\mathrm{ctr}}
=
c_{0,q}\mathsf I_{H_q}
+
c_{1,q}\widetilde m_{\Delta,q}\otimes\widetilde m_{\Delta,q},
\]
where
\[
c_{0,q}
=
q+2+\rho\|\widetilde m_{\Delta,q}\|_{H_q}^{2},
\qquad
c_{1,q}
=
\rho(q+4)
+\rho(1-3\rho)\|\widetilde m_{\Delta,q}\|_{H_q}^{2}.
\]
\end{proof}

\subsection{Proof of lemma~\ref{lem:diagonal-coefficients}}
\begin{proof}
By the definition of $\widetilde V_{jkmn}^{(q)}$,
\[
\begin{aligned}
\widetilde V_{jjjj}^{(q)}
&=
\sum_{\ell=1}^{2}\pi_\ell
\Bigl\{
\mathbb E(\zeta_{\ell j}^{4})
+
2a_\ell^2\widetilde\nu_j^{\,2}
\mathbb E(\zeta_{\ell j}^{2})
+
a_\ell^4\widetilde\nu_j^{\,4}
\Bigr\} \\
&=
\sum_{\ell=1}^{2}\pi_\ell
\mathbb E
\left[
\left(
\zeta_{\ell j}^{\,2}
+
a_\ell^{2}\widetilde\nu_j^{\,2}
\right)^2
\right].
\end{aligned}
\]
Since
\[
\mathbb E(\zeta_{\ell j}^{2})=1,
\qquad
\mathbb E(\zeta_{\ell j}^{4})=3,
\]
it follows that
\[
\widetilde V_{jjjj}^{(q)}
=
3
+
2
\left(
\sum_{\ell=1}^{2}\pi_\ell a_\ell^2
\right)
\widetilde\nu_j^{\,2}
+
\left(
\sum_{\ell=1}^{2}\pi_\ell a_\ell^4
\right)
\widetilde\nu_j^{\,4}.
\]
Using \eqref{eq1} and \eqref{eq2} gives
\[
\widetilde V_{jjjj}^{(q)}
=
3
+
2\rho\,\widetilde\nu_j^{\,2}
+
\rho(1-3\rho)\widetilde\nu_j^{\,4}.
\]
\end{proof}

\subsection{Proof of theorem~\ref{T1}}
\begin{proof}
Since
$
\beta_q
=
\mathsf K_q^{\dagger/2}\widetilde m_{\Delta,q},
$
its expansion in the Karhunen--Loève basis is
\[
\beta_q
=
\sum_{j=1}^{q}
\frac{\widetilde\nu_j}{\sqrt{\lambda_j}}
\gamma_j.
\]
By construction,
$
\mathsf{\widetilde K}_{4,q}\gamma_j
=
\kappa_{j,q}\gamma_j,
$
where
$
\kappa_{j,q}
=
3
+
2\rho\widetilde\nu_j^{\,2}
+
\rho(1-3\rho)\widetilde\nu_j^{\,4}.
$
Hence
\[
\mathsf{\widetilde K}_{4,q}\beta_q
=
\sum_{j=1}^{q}
\kappa_{j,q}
\frac{\widetilde\nu_j}{\sqrt{\lambda_j}}
\gamma_j.
\]

Suppose first that
$
\beta_q
$
is an eigenfunction of
$
\mathsf{\widetilde K}_{4,q}.
$
Then there exists
$
\kappa_{\Delta,q}\in\mathbb R
$
such that
\[
\mathsf{\widetilde K}_{4,q}\beta_q
=
\kappa_{\Delta,q}\beta_q.
\]
Since
$
\{\gamma_j\}_{j=1}^{q}
$
forms a basis of
$
H_q,
$
comparison of coefficients yields
$
\kappa_{j,q}
=
\kappa_{\Delta,q}
$
for every
$
j
$
such that
$
\widetilde\nu_j\neq0.
$

Define
$
g_\alpha(x)
=
3
+
2\rho x
+
\rho(1-3\rho)x^2,
$
$
x\ge0.
$
Since
$
0<\rho=\alpha(1-\alpha)\le1/4,
$
we have
$
1-3\rho>0
$
and therefore
$
g_\alpha'(x)
=
2\rho
+
2\rho(1-3\rho)x
>
0
$
for every
$
x\ge0.
$
Hence
$
g_\alpha
$
is strictly increasing, so
$
\kappa_{j,q}
=
g_\alpha(\widetilde\nu_j^{\,2})
$
is constant on the active coordinates if and only if there exists
$
c>0
$
such that
$
\widetilde\nu_j^{\,2}
=
c
$
whenever
$
\widetilde\nu_j\neq0.
$

Conversely, if
$
\widetilde\nu_j^{\,2}=c
$
for every active coordinate, then every corresponding basis vector has the common eigenvalue
$
\kappa_{\Delta,q}
=
g_\alpha(c)
=
3
+
2\rho c
+
\rho(1-3\rho)c^2.
$
Consequently,
\[
\mathsf{\widetilde K}_{4,q}\beta_q
=
\sum_{\widetilde\nu_j\neq0}
\kappa_{\Delta,q}
\frac{\widetilde\nu_j}{\sqrt{\lambda_j}}
\gamma_j
=
\kappa_{\Delta,q}\beta_q.
\]

Finally, if
$
\widetilde m_{\Delta,q}
=
\widetilde\nu_{j_\ast}\gamma_{j_\ast}
$
for some
$
\widetilde\nu_{j_\ast}\neq0,
$
then
$
\beta_q
$
is a non-zero multiple of
$
\gamma_{j_\ast},
$
which is an eigenfunction of
$
\mathsf{\widetilde K}_{4,q}.
$
\end{proof}

\bl{\subsection{Proof of proposition~\ref{prop:full-tensor}}
\begin{proof}
By the Feldman--H\'ajek theorem and the Cameron--Martin criterion, $P_1\perp P_2$ if and only if $r_q\to\infty$. If $r_q\to\infty$, then $\mathcal A_q^{\mathrm{full}}\ge3+6\rho\,r_q\to\infty$, since $6\rho>0$. Conversely, if $r_q$ remains bounded, then, since $\sum_{j=1}^{q}\widetilde\nu_j^{\,4}\le r_q^{\,2}$, the sum $\sum_{j=1}^{q}\widetilde\nu_j^{\,4}$ remains bounded as well, and hence so does $\mathcal A_q^{\mathrm{full}}$.
\end{proof}}

\subsection{Proof of theorem~\ref{thm:asymptotic-equivalence}}
\begin{proof}
Since
$
\|m_{\Delta,q}\|_{H(\mathfrak K)}^2
=
\sum_{j=1}^{q}\widetilde\nu_j^{\,2},
$
the assumption
$
m_\Delta\notin H(\mathfrak K)
$
is equivalent to
$
\sum_{j=1}^{\infty}\widetilde\nu_j^{\,2}=\infty,
$
and therefore
$
\Delta_q\to\infty
$
because every term defining $\Delta_q$ is nonnegative. Hence
\[
\frac{\mathcal A_q^{(4)}}{\mathcal A_q^{(F)}}
=
\frac{3+\Delta_q}{1+\Delta_q}
=
1+\frac{2}{1+\Delta_q}
\longrightarrow1,
\]
which proves the first statement. The asymptotic equivalence
$
\mathcal A_q^{(4)}
\sim
\mathcal A_q^{(F)}
$
follows immediately.
\end{proof}

\subsection{Proof of theorem~\ref{T2}}
\begin{proof}
By the Feldman--Hájek theorem,
\[
P_1\perp P_2
\iff
m_\Delta\notin H(\mathfrak K).
\]
By the spectral characterization of the Cameron--Martin space,
\[
m_\Delta\notin H(\mathfrak K)
\iff
\sum_{j=1}^{\infty}
\frac{\nu_j^2}{\lambda_j}
=
\sum_{j=1}^{\infty}
\widetilde\nu_j^{\,2}
=
\infty.
\]

By definition,
\[
\Delta_q
=
2\rho
\sum_{j=1}^{q}\widetilde\nu_j^{\,2}
+
\rho(1-3\rho)
\sum_{j=1}^{q}\widetilde\nu_j^{\,4}.
\]
Since
\[
\sum_{j=1}^{\infty}\widetilde\nu_j^{\,4}
\le
\left(
\sum_{j=1}^{\infty}\widetilde\nu_j^{\,2}
\right)^2,
\]
the second series converges whenever the first one does. Since $2\rho>0$, it follows that
\[
\Delta_q\longrightarrow\infty
\iff
\sum_{j=1}^{\infty}\widetilde\nu_j^{\,2}
=
\infty.
\]
Hence,
\[
P_1\perp P_2
\iff
\Delta_q\longrightarrow\infty.
\]
Since
$
\mathcal A_q^{(4)}=3+\Delta_q
$
and
$
\mathcal A_q^{(F)}=1+\Delta_q,
$
we have
\[
\frac{\mathcal A_q^{(4)}}{\mathcal A_q^{(F)}}
=
1+\frac{2}{1+\Delta_q}.
\]
Thus,
\[
\frac{\mathcal A_q^{(4)}}{\mathcal A_q^{(F)}}
\longrightarrow1
\iff
\Delta_q\longrightarrow\infty,
\]
which completes the proof.
\end{proof}

\section*{Statements and Declarations}

{\small
\subsection*{Funding}
The author acknowledges the support from the FWO project G013024N.

\subsection*{Competing interests}
The author has no relevant financial or non-financial competing interests to disclose.

\subsection*{Data availability}
No datasets were generated or analyzed in the preparation of this manuscript, and data availability is therefore not applicable.}

\end{document}

%% file: preamble.tex
\usepackage[utf8]{inputenc}
\usepackage{geometry}
\usepackage{amsmath}
\usepackage{amssymb}
\usepackage{amsthm}
\usepackage{comment}
\usepackage[dvipsnames]{xcolor}
\usepackage{tikz}
\usetikzlibrary{arrows.meta,graphs,graphs.standard,quotes,shapes.misc}

\usepackage{caption}
\usepackage{blkarray}
\usepackage{multirow}
\usepackage{hhline}
\usepackage{nicematrix}
\usepackage{arydshln}

\reversemarginpar
\usepackage[
  tickmarkheight=.1em,
  textwidth=6em,
  textsize=tiny,
  backgroundcolor=gray!15,
  linecolor=gray!70,
  bordercolor=gray!70,
  textcolor=black
]{todonotes}

\usepackage[pdftex]{hyperref}
\hypersetup{
    colorlinks=true,
    linkcolor={MyBlue},
    citecolor={magenta},
    urlcolor={MyBlue}
}

\newtheorem{defn0}{Definition}[section]
\newtheorem{prop0}[defn0]{Proposition}
\newtheorem{thm0}[defn0]{Theorem}
\newtheorem{assump0}[defn0]{Assumption}
\newtheorem{lemma0}[defn0]{Lemma}
\newtheorem{corollary0}[defn0]{Corollary}
\newtheorem{example0}[defn0]{Example}
\newtheorem{remark0}[defn0]{Remark}
\newtheorem{conjecture0}[defn0]{Conjecture}

\newtheorem*{assumption*}{Assumption}
\newtheorem{question*}{Question}

\newenvironment{proposition}{\medskip \begin{prop0}}{\end{prop0}}
\newenvironment{theorem}{\medskip \begin{thm0}}{\end{thm0}}

\newenvironment{lemma}{\medskip \begin{lemma0}}{\end{lemma0}}

\newenvironment{remark}{ \medskip\begin{remark0}\rm}{\end{remark0}}

\definecolor{MyBlue}{RGB}{0,101,189} % Pantone 300 (Wie TUMBlau)
\definecolor{MyRed}{RGB}{234, 114, 55} 
\definecolor{MyGreen}{RGB}{162,173,0}
\definecolor{MyYellow}{RGB}{246, 235, 97} %Pantone 100 
\newcommand{\bl}[1]{{\color{MyBlue}#1}}

\usepackage[nameinlink,capitalise]{cleveref}